\documentclass[12pt,twoside]{amsart}
\usepackage{latexsym,amsfonts,amsmath,amssymb,color}

\newtheorem{Th}{Theorem}[section]
\newtheorem{Rem}[Th]{Remark}
\newtheorem{Ex}[Th]{Example}

\newtheorem{Def}[Th]{Definition}
\newtheorem{Prop}[Th]{Proposition}

\catcode`\@=11

\renewcommand{\section}%
   {\setcounter{equation}{0}\@startsection {section}{1}{\z@}{-3.5ex plus -1ex
  minus -.2ex}{2.3ex plus .2ex}{\Large\bf}}

\def\supp{\mathop{\rm supp}\nolimits}

\def\ds{\displaystyle}
\def\R{\mathbb R}
\def\C{\mathbb C}
\def\N{\mathbb N}

\def\proj{\mathop{\mbox{\rm proj}}}
\def\ind{\mathop{\mbox{\rm ind}}}
\def\dist{\mathop{\mbox{\rm dist}}}

\newcommand{\E}{\mathcal{E}}

\newcommand{\beqsn}{\arraycolsep1.5pt\begin{eqnarray*}}
\newcommand{\eeqsn}{\end{eqnarray*}\arraycolsep5pt}
\newcommand{\beqs}{\arraycolsep1.5pt\begin{eqnarray}}
\newcommand{\eeqs}{\end{eqnarray}\arraycolsep5pt}

\title{A simple proof of Kotake-Narasimhan theorem in some classes of ultradifferentiable
functions}

\author{Chiara Boiti}
\address{
Dipartimento di Matematica e Informatica \\Universit\`a di Ferrara\\
Via Ma\-chia\-vel\-li n.~30\\
I-44121 Ferrara\\
Italy}
\email{chiara.boiti@unife.it}

\author{David Jornet}
\address{
Instituto Universitario de Matem\'atica Pura y Aplicada IUMPA\\
Universitat Po\-li\-t\`ecni\-ca de Val\`encia\\
Camino de Vera, s/n\\
E-46071 Valencia\\
Spain}
\email{djornet@mat.upv.es}

\begin{document}

\begin{abstract}
We give a simple proof of a general theorem of Kotake-Narasimhan for elliptic operators
in the setting of ultradifferentiable functions in the sense of Braun, Meise and
Taylor. We follow the ideas of Komatsu. Based on an example of M\'etivier, we also show that the ellipticity is a necessary condition for the theorem to be true.

The present new version of the paper modifies the proof of Theorem 1.4 for an observation by Hoepfner and Rampazo who pointed out that an induction hypothesis depends on a constant $C_q$ that changes in the induction process, and hence the argument might not work as it was written. However, the statement of the result was originally correct and  modifying $C_q$ with a more concrete expression in the induction hypothesis, the induction procedure is easily clarified with almost the same proof.

Moreover, we eliminate the condition that the weight is identically zero in the interval [0,1], showing that the statements hold true with very similar arguments.
\end{abstract}

\maketitle
\markboth{\sc A simple proof of Kotake-Narasimhan theorem}
{\sc C.~Boiti and D.~Jornet}

\vspace{3mm}
\noindent {\em keywords\,}: {Iterates of an operator, 
Theorem of Kotake-Narasimhan,
ultradifferentiable functions.}

\noindent {\em 2010 Mathematics Subject Classification\,}:
{Primary: 46E10; Secondary: 46F05}

\section{Introduction and main result}

\label{S1}

The problem of iterates began when Komatsu~\cite{K1} in 1960
characterized
analytic functions $f$ in
terms of the behaviour  of
successive iterates $P(D)^jf$ of the function $f$ for a linear partial differential elliptic
operator $P(D)$ with constant coefficients. He proved that a $C^\infty$ function $f$ is real analytic
in $\Omega$ if and only if for every compact set $K\subset\subset\Omega$
there is a constant $C>0$ such that
\beqsn
\|P(D)^jf\|_{L^2(K)}\leq C^{j+1}(j!)^m,
\qquad\forall j\in\N_0:=\N\cup\{0\},
\eeqsn
where $m$ is the order of the operator and $\|\cdot\|_{L^2(K)}$ is the $L^2$ norm
on $K$. This result was generalized to the case of elliptic linear partial 
differential operators $P(x,D)$ with real analytic coefficients in $\Omega$
by Kotake and Narasimhan~\cite{KN}, and is known as 
``the Theorem of Kotake-Narasimhan''. Komatsu~\cite{K3} gave a simpler proof. 
Similar results have been
previously
considered by Nelson~\cite{N}. Later these results were extended to Gevrey 
functions by Newberger and Zielezny~\cite{NZ} in the case of operators with 
constant coefficients. Lions and Magenes~\cite{LM} considered the case of 
Denjoy-Carleman classes of Roumieu type for elliptic linear partial 
differential operators
$P(x,D)$ with variable coefficients in the same Roumieu class, and 
Oldrich~\cite{old} treated the case of Denjoy-Carleman classes of 
Beurling type with some loss of regularity with respect to the 
coefficients. M\'etivier~\cite{metivier1978propiete} proved that the 
result of Lions and Magenes for Gevrey classes is true only for 
elliptic operators in the case of real analytic coefficients.  
Spaces of Gevrey type given by the iterates of a
differential operator are called {\em generalized Gevrey classes} and 
were used
by Langenbruch~\cite{langenbruch1979P,langenbruch1979Fortsetzung,langenbruch1985on,langenbruch1987bases} for different purposes. 

More recently, Juan-Huguet~\cite{H} extended the
results of Komatsu~\cite{K1}, Newberger and Zielezny \cite{NZ} and
M\'etivier~\cite{metivier1978propiete} to the setting of 
non-qua\-si\-ana\-ly\-tic
classes in the sense of Braun, Meise and Taylor~\cite{BMT} for operators 
with constant coefficients. In
\cite{H}, Juan-Huguet introduced the generalized
spaces of ultradifferentiable functions $\E_{*}^P(\Omega)$ on an open subset
$\Omega$ of $\R^{n}$ for a fixed linear partial differential operator $P$
with constant coefficients, and proved that these spaces are complete if
and only if  $P$ is hypoelliptic. Moreover, Juan-Huguet showed that, in this
case, the spaces are nuclear. Later, the same author in \cite{J} established
a Paley-Wiener theorem for the classes $\E^{P}_{*}(\Omega)$, again under
the hypothesis of the hypoellipticity of $P$.

We used in \cite{BJJ} and \cite{BJ2} the results of Juan Huguet to define and characterize a wave front set for the generalized spaces of ultradifferentiable functions $\E_{*}^P(\Omega)$ when $P$ is hypoelliptic. 
In particular, for $P$ elliptic we obtain a microlocal version of the theorem of Kotake and Narasimhan.
In order to remove the assumption on the hypoellipticity of the operator,
we considered in \cite{BJ1} a different setting of ultradifferentiable functions,
following the ideas of \cite{BCM}.

Here, we give a simple proof of the theorem of Kotake-Narasimhan~\cite[Theorem 1]{KN} in the setting of ultradifferentiable functions as introduced by Braun, Meise and Taylor~\cite{BMT} for quasianalytic or non-quasianalytic weight functions. We will consider subadditive weight functions, or more generally, weight functions which satisfy condition $(\alpha_0)$, that we define later (see for example Petzsche and Vogt~\cite[p.~19]{PV} or Fern\'andez and Galbis~\cite[p.~401]{FG}). We follow the lines of Komatsu~\cite{K3}.

Let us recall from \cite{BMT} the definitions of weight functions $\omega$
and of the spaces of ultradifferentiable functions
of Beurling and Roumieu type:

\begin{Def}
\label{defweight}
A non-quasianalytic {\em weight function} is a continuous increasing function
$\omega:\ [0,+\infty[\to[0,+\infty[$ with the following properties:
\begin{itemize}
\item[$(\alpha)$]
$\exists\ L>0$ s.t. $\omega(2t)\leq L(\omega(t)+1)\quad\forall t\geq0$;
\item[$(\beta)$] $\int_1^{+\infty}\frac{\omega(t)}{t^2}dt<+\infty,$

\item[$(\gamma)$]
$\log(t)=o(\omega(t))$ as $t\to+\infty$;
\item[$(\delta)$]
$\varphi_\omega:\ t\mapsto\omega(e^t)$ is convex.
\end{itemize}

We say that $\omega$ is {\em quasianalytic} if, instead of $(\beta)$ it satisfies:
\begin{itemize}
\item[($\beta'$)]
$\ds\int_1^{+\infty}\frac{\omega(t)}{t^2}dt=+\infty$.
\end{itemize}

We will consider also the following property:
\begin{itemize}
\item[$(\alpha_0)$] $\exists\,C>0, \ \exists\, t_0>0,\ \forall\, \lambda \ge 1,\  \forall\, t\ge t_0:\, \omega(\lambda t) \le  \lambda C \omega(t). $
\end{itemize}

\end{Def}
The property $(\alpha_0)$ above is used in \cite[p.~19]{PV} and \cite[p.~401]{FG}, for instance. Moreover, a weight function $\omega$ satisfies $(\alpha_0)$ if and only if it is equivalent to a \emph{subadditive} weight function \cite[Proposition 1.1]{PV}. In the following, we will assume that our weight functions satisfy $(\alpha_0)$, and there is no loss of generality to consider only subadditive weights. This condition should be compared with \cite[(1.4), p.~3]{LM} or \cite[(2), p.~1]{old}, which is a similar condition for Denjoy-Carleman classes.

Normally, we will denote $\varphi_{\omega}$ simply by $\varphi$.

For a weight function $\omega$ we define
$\overline{\omega}:\C^n\rightarrow[0,+\infty[$ by
$\overline{\omega}(z):=\omega(|z|)$ and again we denote this function by
$\omega$.

The {\em Young conjugate} $\varphi^*:\ [0,+\infty[\to
[0,+\infty[$ is defined by
$$
\varphi^*(s):=\sup_{t\geq0}\{st-\varphi(t)\}.
$$
Note that $\varphi^*$ is increasing and
convex, $\varphi^*(t)/t$
is increasing and tends to $\infty$ as $t\rightarrow \infty$, and
$\varphi^{**}=\varphi$.

\begin{Ex}{\rm The following functions are, after a
change
in some interval $[0,M]$,  examples of weight functions:\\
\noindent (i) $\omega (t)=t^d$ for $0<d<1.$ \\ \noindent (ii)
$\omega (t)=\left(\log(1+t)\right)^s$, $s>1.$ \\ \noindent (iii)
$\omega(t)=t(\log(e+t))^{-\beta}$, $\beta>1.$ \\ \noindent (iv)
$\omega(t)=\exp(\beta (\log(1+t))^{\alpha})$, $0<\alpha<1.$}
\end{Ex}
In what follows, $\Omega$ denotes an arbitrary subset of $\R^n$ and
$K\subset\subset\Omega$ means that $K$ is a compact subset in $\Omega$.

\begin{Def}\label{ultradifclases}
{\rm Let $\omega$  be a weight function. For a compact subset $K$ in $\R^n$ which coincides with the closure of
its interior and $\lambda>0$, we define the seminorm
$$
p_{K,\lambda}(f):=\displaystyle\sup_{\alpha\in
    \N_0^n}\sup_{x\in K}\left|f^{(\alpha)}(x)\right|
\exp\left(-\lambda\varphi^*\left(\frac{|\alpha|}{\lambda}\right)\right),
$$
where $\N_0:=\N\cup\{0\}$,
and set
$$\mathcal{E}_{\omega}^{\lambda}(K):= \{ f \in C^\infty(K):
p_{K ,\lambda}(f)<\infty\},$$
which is a Banach space endowed with the $p_{K,\lambda}(\cdot)$-topology.

For an open subset $\Omega$ in $\R^n$, the class of
\emph{$\omega$-ultradifferentiable functions of Beurling type} is defined by
$$\mathcal{E}_{(\omega)}(\Omega):= \{ f \in C^\infty(\Omega):
p_{K ,\lambda}(f)<\infty , \, \mbox{for every} \,
K\subset \subset \Omega \   \mbox{and every}\,  \lambda
>0\}.$$
The topology of this space is
$$ \mathcal{E}_{(\omega)}(\Omega)=
\proj_{\stackrel{\longleftarrow}{K\subset\subset\Omega}}
\proj_{\stackrel{\longleftarrow}{\lambda>0}}
\mathcal{E}_{\omega}^{\lambda}(K),$$
and one can show that $\mathcal{E}_{(\omega)}(\Omega)$ is a Fr\'{e}chet space.

For an open subset $\Omega$ in $\R^n$, the class of
{\it $\omega$-ultradifferentiable functions of Roumieu type} is defined by:
$$\mathcal{E}_{\{\omega\}}(\Omega):= \{ f \in
C^\infty(\Omega):\,\forall  K\subset \subset \Omega\mbox{ } \exists \lambda
>0\mbox{ such that }  p_{K,\lambda}(f)<\infty\}.$$
Its topology is the following
$$
\mathcal{E}_{\{\omega\}}(\Omega)=
\proj_{\stackrel{\longleftarrow}{K\subset\subset\Omega}}\ind_{\stackrel{\longrightarrow}{m\in\mathbb{N}}}\mathcal
{E}_{\omega}^{\frac{1}{m}}(K).
$$
This is a complete PLS-space, that is, a complete space which is
a projective limit of LB-spaces. Moreover,
$\mathcal{E}_{\{\omega\}}(\Omega)$ is also a nuclear and reflexive locally
convex space. In particular, $\mathcal{E}_{\{\omega\}}(\Omega)$ is an
ultrabornological (hence barrelled and bornological) space. }
\end{Def}

The elements of
${\mathcal E}_{(\omega )}(\Omega)$ {\rm (}resp. ${\mathcal
E}_{\{\omega\}}(\Omega)${\rm )} are called ultradifferentiable functions
of Beurling type {\rm (}resp. Roumieu type{\rm )} in $\Omega.$

In the case that $\omega (t):=
t^d$ ($0<d<1$), the corresponding Roumieu class is the Gevrey class with
exponent $1/d.$ In the limit case $d=1$, the
corresponding Roumieu class ${\mathcal
E}_{\{\omega\}}(\Omega)$ is the space of real analytic functions on $\Omega$
whereas the Beurling class ${\mathcal
E}_{(\omega)}({\mathbb R}^n)$ gives the entire functions. Observe that Gevrey weights satisfy $(\alpha_0)$.

Given a polynomial $P\in\mathbb{C}[z_1,\ldots,z_n]$ of degree $m$,
$P(z)=\sum\limits_{|\alpha|\leq m}a_{\alpha}z^{\alpha},$
the partial  differential operator $P(D)$ is defined as
$
P(D)=\sum_{|\alpha|\leq m}a_{\alpha}D^{\alpha}$, where $D=\frac{1}{i}\partial.$ Following \cite{H}, we consider smooth functions in an
open set $\Omega$ such that there exists $C>0$ verifying for each
$j \in {\N}_{0}:=\N\cup \{0\},$
$$
\|P^j(D)f\|_{L^2(K)}\leq C\exp\left(\lambda\varphi^*(\frac{jm}{\lambda})\right),
$$
where $K$ is a compact subset in $\Omega$, $\|\cdot\|_{L^2(K)}$ denotes the
${L}^2$-norm on $K$ and $P^j(D)$ is the $j$-th iterate of the partial
differential operator $P(D)$ of order $m$, i.e.,
 $$P^j(D)=\underbrace{P(D)\circ\cdot\cdot\cdot\circ P(D)}_{j}.$$
 If $j=0$, then we set $P^0(D)f=f.$

The spaces of ultradifferentiable functions with respect to the successive
iterates of $P$ are defined as follows.

 Let $\omega$ be a weight function. Given a polynomial $P$, an open set
$\Omega$ of
$\R^n$, a compact subset $K\subset\subset\Omega$ and $\lambda>0$,
we define the seminorm
\begin{equation}
\|f\|_{K,\lambda}:=\sup_{j\in\mathbb{N}_{0}}\|P^j(D)f\|_{2,K}\exp\left(-\lambda
\varphi^*(\frac{jm}{\lambda})\right)\label{generalized-seminorm}
\end{equation}
and set
$$
\mathcal{E}_{P,\omega}^{\lambda}(K)=\{ f\in C^{\infty}(K):\mbox{ }
\|f\|_{K,\lambda}<+\infty\}.
$$

It is a normed space endowed with the $\|\cdot\|_{K,\lambda}$-norm.

The space of {\it ultradifferentiable functions of Beurling type with
respect to the iterates of $P$} is:
$$
\mathcal{E}^P_{(\omega)}(\Omega)=\{ f\in C^{\infty}(\Omega):
\mbox{ }\|f\|_{K,\lambda}<+\infty\mbox{ for each }K\subset\subset\Omega
\mbox{ and }\lambda>0\},
$$
{\rm endowed with the topology given by}
 $$\mathcal{E}^P_{(\omega)}(\Omega):=
\proj_{\stackrel{\longleftarrow}{K\subset\subset\Omega}}
\proj_{\stackrel{\longleftarrow}{\lambda>0}}\mathcal
{E}_{P,\omega}^{\lambda}(K). $$

{\rm If $\{K_n\}_{n\in\mathbb{N}}$ is a compact
exhaustion of $\Omega$ we have}
 $$ \mathcal{E}^P_{(\omega)}(\Omega)=\proj_{\stackrel{\longleftarrow
}{n\in\mathbb{N}} }\proj_{ \stackrel{\longleftarrow}
 {k\in\mathbb{N}}}\mathcal{E}_{P,\omega}^{k}(K_n)=
\proj_{\stackrel{\longleftarrow}{n\in\mathbb{N}}}\mathcal
{E}_{P,\omega}^{n}(K_n).$$

This is a metrizable locally convex topology defined by the
fundamental system of seminorms
$\left\{\|\cdot\|_{K_{n},n}\right\}_{n\in\mathbb{N}}$. 

The space of {\it ultradifferentiable functions of Roumieu type
with respect to the iterates of $P$} is defined by:
$$
\mathcal{E}^P_{\{\omega\}}(\Omega)=\{ f\in C^{\infty}(\Omega):
\mbox{ }\forall K\subset\subset\Omega\mbox{ }\exists\lambda>0
\mbox{ such that }
\|f\|_{K,\lambda}<+\infty\}.
$$
{\rm Its topology is defined by}
 $$\mathcal{E}^P_{\{\omega\}}(\Omega):=
\proj_{\stackrel{\longleftarrow}{K\subset\subset\Omega}}
\ind_{\stackrel{\longrightarrow}{\lambda>0}}\mathcal
{E}_{P,\omega}^{\lambda}(K). $$

In the following, $\ast$ will denote either $\{\omega\}$ or $(\omega)$. 

The inclusion map $\E_\ast(\Omega)\hookrightarrow \E^P_\ast(\Omega)$ is
continuous (see \cite[Theorem 4.1]{H}). The space
$\mathcal{E}^P_{*}(\Omega)$ is complete if and only if $P$ is hypoelliptic
(see \cite[Theorem 3.3]{H}). Moreover, under a mild
condition on $\omega$ introduced by Bonet, Meise and Melikhov
\cite[16 Corollary (3)]{bonet_meise_melikhov2007a}, $\mathcal{E}^P_{*}(\Omega)$  coincides
with the class of ultradifferentiable functions $\mathcal{E}_{*}(\Omega)$ if
and only if $P$ is elliptic (see \cite[Theorem 4.12]{H}).

Now, let $P(x,D)=\sum_{|\alpha|\le m}a_\alpha(x) D^\alpha$ be a linear partial differential operator of order $m$
with smooth coefficients in an open subset $\Omega\subseteq\R^n$, i.e. $a_\alpha\in C^\infty(\Omega)$ for all multi-index $\alpha\in\N_0^n$ with $|\alpha|\le m$. We consider the
$q$-th iterates $P^q=P\circ{\cdots}\circ P$ of $P:=P(x,D)$ and define the corresponding spaces of iterates as above:
\beqs
\nonumber
\E^P_{(\omega)}(\Omega):=
\{u\in C^\infty(\Omega):\ &&\forall K\subset\subset\Omega\, 
\forall k\in\N\ \exists c_k>0\ \mbox{s.t.}\\
\label{beurling}
&&\|P^q u\|_{L^2(K)}\leq c_ke^{k\varphi^*(qm/k)}\ \forall q\in\N_0\}
\eeqs
for the Beurling case, and
\beqs
\nonumber
\E^P_{\{\omega\}}(\Omega):=
\{u\in C^\infty(\Omega):\ &&\forall K\subset\subset\Omega\,
\exists
 k\in\N,\,  c>0\ \mbox{s.t.}\\
\label{roumieu}
&&\|P^q u\|_{L^2(K)}\leq ce^{\frac1k\varphi^*(qmk)}\ \forall q\in\N_0\}
\eeqs
for the Roumieu case.

We generalize some results of Juan-Huguet~\cite{H} for operators with variable coefficients in the following way. First, we state our main result in the Roumieu case:
\begin{Th}
\label{th13bis}
Let $\omega$ be a subadditive weight function, $\Omega\subseteq \R^n$ a domain, i.e. open and connected,
 and $P(x,D)$ a linear partial differential operator of 
order $m$ with coefficients in $\E_{\{\omega\}}(\Omega)$. Then:
\begin{itemize}
\item[(i)]
$\E_{\{\omega\}}(\Omega)\subseteq\E^P_{\{\omega\}}(\Omega)$;
\item[(ii)]
if $P(x,D)$ is elliptic, then $\E_{\{\omega\}}(\Omega)=\E^P_{\{\omega\}}(\Omega)$.
\end{itemize}
\end{Th}

In the Beurling case we lose some regularity; compare to Oldrich~\cite[Teorema 1]{old}:

\begin{Th}
\label{th13}
Let $\omega$ be a subadditive weight function, $\Omega\subseteq \R^n$ a domain 
and $P(x,D)$ a linear partial differential operator of 
order $m$ with coefficients in $\E_{(\omega)}(\Omega)$. Then:
\begin{itemize}
\item[(i)]
$\E_{(\omega)}(\Omega)\subseteq\E^P_{(\omega)}(\Omega)$;
\item[(ii)]
if $P(x,D)$ is elliptic, then 
$\E^P_{(\omega)}(\Omega)\subseteq
\E_{(\sigma)}(\Omega)$ for every subadditive weight function 
$\sigma(t)=o(\omega(t))$ as $t\to+\infty$.
\end{itemize}
\end{Th}

Theorem \ref{th13bis} is the generalization to the class of ultradifferentiable functions
$\E_{\{\omega\}}(\Omega)$ of the theorem of Kotake-Narasimhan for an elliptic linear partial 
differential operator $P(x,D)$ with coefficients in the same class $\E_{\{\omega\}}(\Omega)$. We observe that the ellipticity of $P$ is not needed for the inclusion $\E_{\{\omega\}}(\Omega)\subseteq\E^P_{\{\omega\}}(\Omega)$. 
 However, we show in Example \ref{ex32} that the ellipticity is necessary for the equality $\E_{\{\omega\}}(\Omega)=\E^P_{\{\omega\}}(\Omega)$ for a large family of weights $\omega$. We use the example of Metivier~\cite[p.~831]{metivier1978propiete} to show that for suitable weight functions, which \emph{are not} of Gevrey type in general, indeed weights which are between two given concrete Gevrey weights, statement (ii) in Theorems \ref{th13bis} and \ref{th13} fails if $P$ is not elliptic. Finally, we remark that there is no restriction to assume that the weight $\omega$ is quasianalytic, i.e. satisfies condition $(\beta')$ and not $(\beta)$, in Theorems \ref{th13bis} and \ref{th13}. However, in Example \ref{ex32} the weights are taken to be non-quasianalytic.

\section{Preliminary results}
In order to prove Theorems~\ref{th13bis} and \ref{th13}
we collect in this section some preliminary results. First of all, we shall prove some properties of the Young conjugate function
$\varphi^*$ defined in Section~\ref{S1}:

\begin{Prop}
\label{prop1}
Let $\omega$ be a subadditive weight function and define, for $j\in\N_0$,
$\lambda>0$,
\beqsn
a_{j,\lambda}:=\frac{e^{\lambda\varphi^*(j/\lambda)}}{j!}.
\eeqsn
Then the following properties are satisfied:
\begin{itemize}
\item[(a)]
$a_{j,\lambda}\cdot a_{h,\lambda}\leq a_{j+h,\lambda}\qquad\forall j,h\in\N_0,
\,\lambda>0$;
\item[(b)]
$a_{j,\lambda}\leq e^{-\lambda\varphi^*(1/\lambda)}a_{j+1,\lambda}\qquad\forall j\in\N_0,
\,\lambda>0$;\\
if $0<\lambda\leq1$ and $j\leq\ell\in\N_0$ then
$a_{j,\lambda}\leq e^{-(\ell-j)\varphi^*(0)}a_{\ell,\lambda}$;
\item[(c)]
$\lambda\mapsto a_{j,\lambda}$ is decreasing for all $j\in\N_0$;
\item[(d)]
$a_{j+h,\lambda}\leq a_{j,\lambda/2}\cdot a_{h,\lambda/2}\qquad\forall j,h\in\N_0,
\,\lambda>0$;
\item[(e)]
for every $\rho,\lambda>0$ there exists $\lambda', D_{\rho,\lambda}>0$ such that
\beqsn
\rho^je^{\lambda\varphi^*(j/\lambda)}\leq D_{\rho,\lambda}
e^{\lambda'\varphi^*(j/\lambda')}\qquad
\forall j\in\N_0,
\eeqsn
with $D_{\rho,\lambda}:=\exp\{\lambda[\log\rho+1]\}$, 
where $[\log\rho+1]$ is the integer
part of $\log\rho+1$;
\item[(f)]
  for every $j,h,r\in\N_0$ with $0\leq h\leq j$, and for all $\lambda>0$:
  \beqsn
  \frac{j!}{h!}a_{j-h,\lambda}\leq
  \frac{e^{\lambda\varphi^*\left(\frac{j+r}{\lambda}\right)}}{e^{\lambda
\varphi^*\left(\frac{h+r}{\lambda}\right)}};
  \eeqsn
\item[(g)]
  for every $j,h,r\in\N_0,\,\lambda>0$:
  \beqsn
  e^{\lambda\varphi^*\left(\frac{j}{\lambda}\right)}e^{\lambda\varphi^*
\left(\frac{r+h}{\lambda}\right)}
  \leq e^{\frac \lambda2\varphi^*\left(\frac{j+h}{\lambda/2}\right)}e^{\frac 
\lambda2\varphi^*\left(\frac{r}{\lambda/2}\right)}.
  \eeqsn
  \item[(h)]
  for every $\lambda>0$ and $q,r\in\N_0$ with $q\geq r$ we have that
  \beqsn
  \frac{e^{\lambda\varphi^*\left(\frac{q+1}{\lambda}\right)}}{e^{\lambda\varphi^*
  \left(\frac{q}{\lambda}\right)}}\geq
  \frac{e^{\lambda\varphi^*\left(\frac{r+1}{\lambda}\right)}}{e^{\lambda\varphi^*
  \left(\frac{r}{\lambda}\right)}}\,.
  \eeqsn
  \end{itemize}
\end{Prop}

\begin{proof}
  $(a)$ has been proved in Lema 3.2.3 of \cite{david}.

 The first inequality of $(b)$ follows from $(a)$ since $a_{1,\lambda}=e^{\lambda\varphi^*(1/\lambda)}$.
 
 Iterating the inequality for $0<\lambda\leq1$ and $j\leq\ell$ we get that
 \beqsn
 a_{j,\lambda}\leq e^{-(\ell-j)\lambda\varphi^*(1/\lambda)}a_{\ell,\lambda}
 \leq e^{-(\ell-j)\varphi^*(1)}a_{\ell,\lambda}\leq e^{-(\ell-j)\varphi^*(0)}a_{\ell,\lambda},
 \eeqsn
 since $\varphi^*(s)/s$ and $\varphi^*$ are increasing.
  
  $(c)$ follows from the fact that $\varphi^*(s)/s$ is increasing (cf.
  \cite{BMT}).

  $(d)$ follows from the convexity of $\varphi^*$:
  \beqsn
  a_{j+h,\lambda}=\frac{e^{\lambda\varphi^*\left(\frac{j+h}{\lambda}\right)}}{(j+h)!}\leq
  \frac{j!h!}{(j+h)!}\frac{e^{\frac \lambda2\varphi^*\left(\frac{2j}{\lambda}\right)}}{j!}
  \frac{e^{\frac \lambda2\varphi^*\left(\frac{2h}{\lambda}\right)}}{h!}\\
  =\frac{1}{\binom{j+h}{h}}a_{j,\frac \lambda2}a_{h,\frac \lambda2}
  \leq a_{j,\frac \lambda2}a_{h,\frac \lambda2}.
  \eeqsn

 Point $(e)$ follows from the next property of 
\cite[Prop. 0.1.5(2)(a)]{david}:  for each $y\geq0,$ $n\in\N$, and $\lambda>0$,
  \beqs
  \label{015}
  \lambda L^n\varphi^*\left(\frac{y}{\lambda L^n}\right)+n y
  \leq \lambda\varphi^*\left(\frac y\lambda\right)+\lambda\sum_{h=1}^n L^h,
   \eeqs
  where $L>0$ is such that $\omega(et)\leq L(1+\omega(t))$ for all
  $t\geq0$ (in our case $\omega$ is increasing and subadditive, so that we
  can take $L=3$).

  Indeed, from \eqref{015} with $y=jL^n$ and dividing by $L^n$:
  \beqsn
  \lambda\varphi^*\left(\frac j\lambda\right)+nj\leq
\frac{\lambda}{L^n}
  \varphi^*\left(\frac{j}{\lambda/L^n}\right)+\lambda
  \sum_{h=1}^nL^{h-n}
  \eeqsn
  and therefore
  \beqsn
  \rho^je^{\lambda\varphi^*\left(\frac j\lambda\right)}\leq
  e^{\frac{\lambda}{L^n}\varphi^*\left(\frac{j}{\lambda/L^n}\right)+
\lambda n-nj+j\log\rho}.
  \eeqsn

  Choosing $n_{\rho}:=[\log\rho+1]\in\N$ so that 
$-n_\rho
  +\log\rho\leq0$, for $\lambda'=\lambda/L^{n_{\rho}}$ we thus have that
  \beqs
\label{BJO}
  \rho^je^{\lambda\varphi^*\left(\frac j\lambda\right)}\leq e^{\lambda n_{\rho}}
  e^{\lambda'\varphi^*\left(\frac{j}{\lambda'}\right)}
  \eeqs
  so that $(e)$ is proved.

  In order to prove $(f)$, let us first remark that
  \beqs
  \label{L81}
  \frac{j!}{h!}a_{j-h,\lambda}\leq\frac{(j+r)!}{(h+r)!}a_{j-h,\lambda}
  \eeqs
  since 
$h\leq j$.

  From \eqref{L81} we have that
  \beqsn
  \frac{j!}{h!}a_{j-h,\lambda}\leq&&\frac{(j+r)!}{e^{\lambda\varphi^*
\left(\frac{j+r}{\lambda}\right)}}\cdot
  \frac{e^{\lambda\varphi^*\left(\frac{h+r}{\lambda}\right)}}{(h+r)!}\cdot
  \frac{e^{\lambda\varphi^*\left(\frac{j+r}{\lambda}\right)}}{e^{\lambda\varphi^*
\left(\frac{h+r}{\lambda}\right)}}
  a_{j-h,\lambda}\\
  =&&\frac{a_{h+r,\lambda}\, a_{j-h,\lambda}}{a_{j+r,\lambda}}\cdot
  \frac{e^{\lambda\varphi^*\left(\frac{j+r}{\lambda}\right)}}{e^{\lambda\varphi^*
\left(\frac{h+r}{\lambda}\right)}}
  \leq\frac{e^{\lambda\varphi^*\left(\frac{j+r}{\lambda}\right)}}{e^{\lambda\varphi^*
\left(\frac{h+r}{\lambda}\right)}}
  \eeqsn
  by the already proved point $(a)$.
  Therefore $(f)$ holds true.

  Property $(g)$ follows from the convexity of $\varphi^*$.
  Indeed, from $(a)$
  \beqsn
  e^{\lambda\varphi^*\left(\frac{j}{\lambda}\right)}e^{\lambda\varphi^*
\left(\frac{r+h}{\lambda}\right)}=&&
  a_{j,\lambda}\, a_{r+h,\lambda}\, j!(r+h)!\\
  \leq && a_{j+r+h,\lambda}\, j!(r+h)!
  =e^{\lambda\varphi^*\left(2\frac{j+r+h}{2\lambda}\right)}
  \frac{j!(r+h)!}{(j+r+h)!}\\
  \leq&&e^{\frac \lambda2\varphi^*\left(\frac{j+h}{\lambda/2}\right)+
\frac \lambda2\varphi^*\left(\frac{r}{\lambda/2}\right)}
  \frac{1}{\binom{j+r+h}{j}}\\
  \leq && e^{\frac \lambda2\varphi^*\left(\frac{j+h}{\lambda/2}\right)}
  e^{\frac \lambda2\varphi^*\left(\frac{r}{\lambda/2}\right)}.
  \eeqsn
  
  Let us finally prove $(h)$. We first remark that, by the convexity of $\varphi^*$,
  \beqsn
  2\varphi^*\left(\frac{r+1}{\lambda}\right)=2\varphi^*\left(\frac{r}{2\lambda}
  +\frac{r+2}{2\lambda}\right)
  \leq\varphi^*\left(\frac{r}{\lambda}\right)+\varphi^*\left(\frac{r+2}{\lambda}\right)
  \eeqsn
  i.e.
  \beqsn
  \varphi^*\left(\frac{r+1}{\lambda}\right)-\varphi^*\left(\frac{r}{\lambda}\right)
  \leq\varphi^*\left(\frac{r+2}{\lambda}\right)-\varphi^*\left(\frac{r+1}{\lambda}\right).
  \eeqsn
  Arguing recursively we get
  \beqs
  \label{qr}
  \varphi^*\left(\frac{r+1}{\lambda}\right)-\varphi^*\left(\frac{r}{\lambda}\right)
  \leq\varphi^*\left(\frac{q+1}{\lambda}\right)-\varphi^*\left(\frac{q}{\lambda}\right)
  \eeqs
  for every $q\in\N$ with $q\geq r$.
  
  Clearly \eqref{qr} implies $(h)$ and the proof is complete.
\end{proof}

\begin{Rem}
\begin{em}
Note that we did not use the subadditivity of the weight $\omega$ to
prove points $(c),(d),(e),(h)$ of Proposition~\ref{prop1}.
\end{em}
\end{Rem}

For the proof of Theorem~\ref{th13bis}
we shall follow the ideas of \cite{K3},
so we define, for a domain $\Omega\subseteq\R^n$,  $q\in\N_0,\delta>0$ and $f\in C^\infty(G)$, with $G$
a relatively compact subdomain of $\Omega$,
\beqsn
\|\nabla^qf\|_\delta=\sum_{|\alpha|=q}\|D^\alpha f\|_{L^2(G_\delta)},
\eeqsn
where
\beqsn
G_\delta:=\{x\in G:\ \dist(x,\partial G)>\delta\}
\eeqsn
and $\|\cdot\|_{L^2(G_\delta)}=0$ if $G_\delta=\emptyset$.

If $P=P(x,D)$ is an elliptic linear partial differential operator of
order $m$ with $C^\infty$ coefficients, then the following a priori
estimates, for $\delta,\sigma>0$ and $0\leq r\leq m$, have been
proved in \cite{K2}:
\beqs
\label{K2}
&&\|\nabla^mf\|_{\delta+\sigma}\leq C(\|Pf\|_\sigma+
\delta^{-m}\|f\|_\sigma)\\
\label{K3}
&&\|\nabla^{m-r}f\|_{\delta+\sigma}\leq C\varepsilon^r
(\|\nabla^mf\|_\sigma+(\delta^{-m}+\varepsilon^{-m})\|f\|_\sigma),
\eeqs
for arbitrary $\varepsilon>0$, where the constant $C>0$ depends only on
the operator $P$ and the set $G$.

Then we define the semi-norm $N^{pm}(u)$ by
\beqsn
N^{pm}(u):=\sup_{0<\delta\leq1}\delta^{pm}\|\nabla^{pm}u\|_\delta.
\eeqsn

The following inequality holds:
\begin{Prop}
  \label{prop5bis}
  Let $\Omega\subseteq\R^n$ be a domain and
$P(x,D)$ an elliptic linear partial differential operator of
  order $m$ with coefficients in $\E_{\{\omega\}}(\Omega)$. For $u\in C^\infty(\Omega)$, there exist $k\in\N$ and a positive constant $C_{0}$ such that
  \beqs
\label{9}
   N^{pm}(u)\leq C_{0}\left\{N^{(p-1)m}(Pu)
    +\sum_{q=0}^{p-1}
    \frac{e^{\frac1k\varphi^*(pmk)}}{e^{\frac1k\varphi^*(qmk)}}
    N^{qm}(u)\right\}.
    \eeqs
  for every $p\in\N$.
\end{Prop}

\begin{proof} 
  By definition of the semi-norm $N^{(p+1)m}(u)$ and by \eqref{K2} we have
  \beqs
  \nonumber
  N^{(p+1)m}(u)=&&\sup_{(p+2)\delta\leq1}((p+2)\delta)^{(p+1)m}
  \|\nabla^{(p+1)m}u\|_{(p+2)\delta}\\
  \nonumber
  \leq&&\sup_{(p+2)\delta\leq1}\left(\frac{p+2}{p}\right)^{(p+1)m}
  (p\delta)^{(p+1)m}C(\|P\nabla^{pm}u\|_{(p+1)\delta}\\ &&\qquad+\delta^{-m}
  \|\nabla^{pm}u\|_{(p+1)\delta})\nonumber\\
  \label{eq1}
  \leq&&9^mC\sup_{(p+2)\delta\leq1}\{(p\delta)^{(p+1)m}
\|P\nabla^{pm}u\|_{(p+1)\delta}
\nonumber  \\ &&\qquad+p^m(p\delta)^{pm}\|\nabla^{pm}u\|_{(p+1)\delta}\},
  \eeqs
  since $\left(\frac{p+2}{p}\right)^{p+1}\leq9$.

We set $P^{[r]}:=\sum_{|\alpha|=r}\sup_G|D_x^\alpha P|.$ Since $\|\cdot\|_{(p+1)\delta}\leq\|\cdot\|_{p\delta}$ and
  $p^m(pm)!\leq  ((p+1)m)!$, from \eqref{eq1} and Leibniz' formula we get:
  \beqs
  \nonumber
  N^{(p+1)m}(u)\leq&&9^mC\sup_{(p+2)\delta\leq1}\Big\{(p\delta)^{(p+1)m}
  \Big[
    \|\nabla^{pm}Pu\|_{(p+1)\delta}\\
\nonumber
&&+\sum_{r=1}^{pm}\binom{pm}{r}
    \|P^{[r]}\nabla^{pm-r}u\|_{(p+1)\delta}\Big]
  +p^m(p\delta)^{pm}\|\nabla^{pm}u\|_{p\delta}\Big\}\\
  \nonumber
  \leq&&9^mC\sup_{(p+2)\delta\leq1}\left\{\left(\frac{p}{p+1}\right)^{pm}
  [(p+1)\delta]^{pm}\right.\times\\ 
  &&\nonumber\times \left.\left(\frac{p}{p+2}\right)^m[(p+2)\delta]^m
  \|\nabla^{pm}Pu\|_{(p+1)\delta}\right.\\
  \nonumber
  &&\left. +(p\delta)^{(p+1)m}\sum_{r=1}^{pm}\binom{pm}{r}
\|P^{[r]}\nabla^{pm-r}u\|_{(p+1)\delta}\right.\\
\nonumber&& \left. +\frac{((p+1)m)!}{(pm)!}N^{pm}(u)\right\}\\
   \leq&&9^mC\Big\{N^{pm}(Pu)+\nonumber\\ \nonumber
&&   \sup_{(p+2)\delta\leq1}(p\delta)^{(p+1)m}\sum_{r=1}^{pm}\binom{pm}{r}
  \|P^{[r]}\nabla^{pm-r}u\|_{(p+1)\delta}\\
  \label{5}
  &&+\frac{((p+1)m)!}{(pm)!}N^{pm}(u)\Big\}.
  \eeqs

Taking into account that the coefficients of $P(x,D)$ are in 
$\E_{\{\omega\}}(\Omega)$, we can write the following estimates, 
for $(p+2)\delta\leq1$ and for some $k\in\N$ and
$c>0$:
\beqs
\label{eq8}
\sum_{r=1}^{pm}\binom{pm}{r}&&\|P^{[r]}\nabla^{pm-r}u\|_{(p+1)\delta}\leq
c\sum_{r=1}^{pm}\binom{pm}{r}e^{\frac1k\varphi^*(rk)}
\sum_{s=0}^m\|\nabla^{pm+s-r}u\|_{(p+1)\delta}\nonumber\\
\label{eq9}
\leq&&c\sum_{r=1}^{pm}\frac{(pm)!}{(pm-r)!}a_{r,\frac1k}
\sum_{s=0}^m\|\nabla^{pm+s-r}u\|_{(p+1)\delta}.
\eeqs

By the change of indexes $r=(p-q)m+t$ we obtain that (cf. also \cite{K3})
\beqs
\nonumber
\sum_{r=1}^{pm}\binom{pm}{r}\|P^{[r]}\nabla^{pm-r}u\|_{(p+1)\delta}
\leq&&c(m+1)\sum_{q=1}^p\sum_{t=1}^m\frac{(pm)!}{(qm-t)!}a_{(p-q)m+t,\frac1k}\times\\
\nonumber&&\qquad \times
\|\nabla^{(q+1)m-t}u\|_{(p+1)\delta}\\
\nonumber
&&+c m\sum_{t=1}^m(pm)!a_{pm,\frac1k}\|\nabla^{m-t}u\|_{(p+1)\delta}\\
\nonumber
=&&c(m+1)\sum_{t=1}^m\frac{(pm)!}{(pm-t)!}a_{t,\frac1k}\times\\
\nonumber&&\qquad \times\|\nabla^{(p+1)m-t}
u\|_{(p+1)\delta}\\
\nonumber
&&+c(m+1)\sum_{q=1}^{p-1}\sum_{t=1}^m\frac{(pm)!}{(qm-t)!}a_{(p-q)m+t,\frac1k}\times\\
\nonumber &&\qquad \times\|\nabla^{(q+1)m-t}u\|_{(p+1)\delta}\\
\label{eq6}
&&+cm\sum_{t=1}^m(pm)!a_{pm,\frac1k}\|\nabla^{m-t}u\|_{(p+1)\delta}.
\eeqs

From \eqref{eq6}, by properties $(b)$ and $(d)$ of Proposition~\ref{prop1}
we get:
\beqs
\label{eq3}
\sum_{r=1}^{pm}\binom{pm}{r}\|P^{[r]}\nabla^{pm-r}u\|_{(p+1)\delta}\leq
S_1+S_2+S_3
\eeqs
with
\beqsn
&&S_1:=c(m+1)\max\{1,e^{-(m-1)\varphi^*(0)}\}
\sum_{t=1}^m\frac{(pm)!}{(pm-t)!}a_{m,\frac1k}\|\nabla^{(p+1)m-t}
u\|_{(p+1)\delta}\\
&&S_2:=c\max\{1,e^{-(m-1)\varphi^*(0)}\}a_{m,\frac{1}{2k}}(m+1)\sum_{q=1}^{p-1}\sum_{t=1}^m\frac{(pm)!}{(qm-t)!}
a_{(p-q)m,\frac{1}{2k}}\|\nabla^{(q+1)m-t}u\|_{(p+1)\delta}\\
&&S_3:=cm\sum_{t=1}^m(pm)!a_{pm,\frac1k}\|\nabla^{m-t}u\|_{(p+1)\delta},
\eeqsn 
since $a_{t,\frac1k}\leq e^{-(m-t)\varphi^*(0)}a_{m,\frac1k}
\leq\max\{1,e^{-(m-1)\varphi^*(0)}\}a_{m,\frac1k}$.

By property $(c)$ of Proposition~\ref{prop1} and by \eqref{K3}, setting
\beqsn
C_{2}:=9^mcC(m+1)\max\{1,e^{-(m-1)\varphi^*(0)}\}a_{m,\frac{1}{2k}},
\eeqsn
we have the estimate
\beqsn
9^mC(p\delta)^{(p+1)m}S_1\leq&&C_{2}\sum_{t=1}^m\frac{(pm)!}{(pm-t)!}
(p\delta)^{(p+1)m}\|\nabla^{(p+1)m-t}u\|_{(p+1)\delta}\\
\leq&&C_{2}C\sum_{t=1}^m(pm)^t(p\delta)^{(p+1)m}\varepsilon^t
(\|\nabla^{(p+1)m}u\|_{p\delta}\\&&\nonumber\qquad+(\delta^{-m}+\varepsilon^{-m})
\|\nabla^{pm}u\|_{p\delta})\\
=&&C_{2}C\sum_{t=1}^m(pm)^t\varepsilon^t\big\{(p\delta)^{(p+1)m}\|
\nabla^{(p+1)m}u\|_{p\delta}\\
&&+(p^m+(p\delta)^m\varepsilon^{-m})(p\delta)^{pm}\|\nabla^{pm}u\|_{p\delta}\big\},
\eeqsn
since $(pm)!\leq (pm-t)!(pm)^t$.

Therefore, for $\varepsilon=(pm)^{-1}(2mCC_{2})^{-1/t}$ and 
$(p+2)\delta\leq1$:
\beqs
\nonumber
9^mC(p\delta)^{(p+1)m}S_1\leq&&
\sum_{t=1}^m\frac{1}{2m}\Big\{N^{(p+1)m}(u)\\
\nonumber
&&+\Big(p^m+\Big(\frac{p}{p+2}\Big)^m[(p+2)\delta]^m\times \\&&\nonumber\qquad\times(pm)^m
(2mCC_{2})^{m/t}\Big)N^{pm}(u)\Big\}\\
\nonumber
\leq&&\sum_{t=1}^m\frac{1}{2m}\left\{N^{(p+1)m}(u)\right.\\ \nonumber&&\qquad\left.+
\left(p^m+(pm)^m(2mCC_{2})^{m/t}\right)N^{pm}(u)\right\}\\
\nonumber
\leq&&\frac12N^{(p+1)m}(u)+C_{3}p^mN^{pm}(u)\\
\label{2}
\leq&&\frac12N^{(p+1)m}(u)+C_{3}\frac{((p+1)m)!}{(pm)!}N^{pm}(u)
\eeqs
for some $C_{3}>0$, because of $p^m(pm)!\leq((p+1)m)!$.

In order to estimate $S_2$, let us first prove the following estimate, 
for $1\leq q\leq p-1$,
$(p+1)\delta$$=$$(q+1)\delta'$ 
and $(p+2)\delta\leq1$:
\beqs
\label{lemma6}
(p\delta)^{(p+1)m}\leq(2e)^m(q\delta')^{(q+1)m}.
\eeqs
Indeed,
\beqsn
(p\delta)^{(p+1)m}=&&\frac{p^{(p+1)m}\delta^{(p+1)m}}{q^{(q+1)m}
\left(\frac{p+1}{q+1}\right)^{(q+1)m}
\delta^{(q+1)m}}\cdot(q\delta')^{(q+1)m}\\
=&&\left(\frac{p}{p+1}\frac{q+1}{q}\right)^{(q+1)m}(p\delta)^{(p-q)m}
(q\delta')^{(q+1)m}\\
\leq&&\left(1+\frac 1q\right)^{qm}\left(1+\frac 1q\right)^m
\left(\frac{p}{p+2}\right)^{(p-q)m}\times\\\nonumber&&\qquad\times[(p+2)\delta]^{(p-q)m}
(q\delta')^{(q+1)m}\\
\leq&&e^m2^m(q\delta')^{(q+1)m}.
\eeqsn
Therefore \eqref{lemma6} is proved and,
for $1\leq q\leq p-1$, $(p+1)\delta=(q+1)\delta'$ and $(p+2)\delta\leq1$:
\beqsn
&& 9^mC(p\delta)^{(p+1)m}S_2\\
\leq&& C_{2}\sum_{q=1}^{p-1}\sum_{t=1}^m\frac{(pm)!}{(qm-t)!}a_{(p-q)m,\frac{1}{2k}}
(p\delta)^{(p+1)m}\|\nabla^{(q+1)m-t}u\|_{(p+1)\delta}\\
\leq&&(2e)^m\sum_{q=1}^{p-1}\frac{(pm)!}{(qm)!}a_{(p-q)m,\frac{1}{2k}}\,\times\\
&&\times
C_{2}\sum_{t=1}^m\frac{(qm)!}{(qm-t)!}(q\delta')^{(q+1)m}
\|\nabla^{(q+1)m-t}u\|_{(q+1)\delta'}.
\eeqsn

By \eqref{2} with $q$ and $\delta'$ instead of $p$ and $\delta$ respectively,
and because of properties $(f)$ and $(b)$ of Proposition~\ref{prop1}
we finally get the following estimate for $S_2$:
\beqs
\nonumber
&&9^mC(p\delta)^{(p+1)m}S_2\\\nonumber
&& \leq D\sum_{q=1}^{p-1}\frac{(pm)!}{(qm)!}a_{(p-q)m,\frac{1}{2k}}
\Big\{\frac12 N^{(q+1)m}(u)+C'_{3}
\frac{((q+1)m)!}{(qm)!}N^{qm}(u)\Big\}\\
\nonumber
&&\leq D'\sum_{q=1}^{p-1}\left(
\frac{e^{\frac{1}{2k}\varphi^*\left(2(p+1)mk\right)}}{e^{\frac{1}{2k}\varphi^*
\left(2(q+1)mk
\right)}}N^{(q+1)m}+e^{-m\varphi^*(0)}
\frac{e^{\frac{1}{2k}\varphi^*\left(2(p+1)mk\right)}}{e^{\frac{1}{2k}
\varphi^*\left(2qmk
\right)}}N^{qm}(u)\right)\\
\label{6}
&&\leq (1+e^{-m\varphi^*(0)})D'\sum_{q=1}^{p-1}
\frac{e^{\frac{1}{2k}\varphi^*\left(2(p+1)mk\right)}}{e^{\frac{1}{2k}\varphi^*\left(
2qmk
\right)}}N^{qm}(u)
+D'
  \frac{e^{\frac{1}{2k}\varphi^*\left(2(p+1)mk\right)}}{e^{\frac{1}{2k}\varphi^*\left(2
pmk
\right)}}N^{pm}(u)
\eeqs
for some $C'_{3}, D,D'>0$.

Let us now estimate $S_3$. By \eqref{K3} with $\varepsilon=1$ and because of
properties $(e)$, $(f)$ (with $h=0$) and $(b)$ of 
Proposition~\ref{prop1}, for $(p+2)\delta\leq1$:
\beqs
\nonumber
9^mC(p\delta)^{(p+1)m}S_3
\leq&& C_{2}\max\{1,e^{-(m+1)\varphi^*(0)}\}
\sum_{t=1}^m(pm)!a_{pm,\frac1k}(p\delta)^{(p+1)m}
\|\nabla^{m-t}u\|_{(p+1)\delta}\\
\nonumber
\leq &&CC_{2}\max\{1,e^{-(m+1)\varphi^*(0)}\}\\
\nonumber
&&\cdot\sum_{t=1}^m(pm)!(p\delta)^{pm}a_{pm,\frac1k}\big((p\delta)^m
\|\nabla^mu\|_{p\delta}+p^m(1+\delta^m)\|u\|_{p\delta}\big)\\
\nonumber
\leq &&CC_{2}\max\{1,e^{-(m+1)\varphi^*(0)}\}\sum_{t=1}^m(pm)!a_{pm,\frac1k}\left(N^m(u)+2p^m
N^0(u)\right)\\
\nonumber
\leq&& CC_{2}\max\{1,e^{-(m+1)\varphi^*(0)}\}m(pm)!a_{pm,\frac{1}{k}}N^m(u)\\
\nonumber
&&+2CC_{2}\max\{1,e^{-(m+1)\varphi^*(0)}\}m((p+1)m)!a_{pm,\frac{1}{k}}N^0(u)\\
\nonumber
\leq &&CC_{2}\max\{1,e^{-(m+1)\varphi^*(0)}\}m
\frac{e^{\frac{1}{k}\varphi^*\left((p+1)mk\right)}}{e^{\frac{1}{k}
\varphi^*\left(mk\right)}}
N^m(u)\\
\nonumber
&&+2CC_{2}\max\{1,e^{-(2m+1)\varphi^*(0)}\}m((p+1)m)!a_{(p+1)m,\frac{1}{k}}N^0(u)\\
\label{7}
&&\leq \tilde{D}e^{\frac{1}{k}\varphi^*\left((p+1)mk\right)}\left(N^m(u)+N^0(u)\right),
\eeqs
for some $\tilde{D}>0$, since 
$
a_{m,\frac{1}{2k}}\geq e^{m\varphi^*(0)}a_{0,\frac{1}{2k}}
=e^{m\varphi^*(0)}e^{\frac{1}{2k}\varphi^*(0)}
\geq\min\{1,e^{(m+1)\varphi^*(0)}\},
$
and 
$a_{pm,\frac1k}\leq e^{-m\varphi^*(0)}a_{(p+1)m,\frac1k}$.

Substituting \eqref{2}, \eqref{6} and \eqref{7} in \eqref{eq3} and then in
\eqref{5} and applying
$(b)$ of Proposition~\ref{prop1}, we finally get:
\beqsn
N^{(p+1)m}(u)\leq C_5N^{pm}(Pu)+\frac12 N^{(p+1)m}(u)
+C_{5}\sum_{q=0}^p
\frac{e^{\frac{1}{k'}\varphi^*\left((p+1)mk'\right)}}{e^{\frac{1}{k'}
\varphi^*\left(qmk'\right)}}
N^{qm}(u),
\eeqsn
for some $k'\in\N$ and $C_{5}>0$, concluding the proof.
\end{proof}

We shall also need, in the following, the next result:
\begin{Prop}
  \label{prop10bis}
  Let $P(x,D)$ be an elliptic linear partial differential operator of
  order $m$ with coefficients in $\E_{\{\omega\}}(\Omega)$. For $u\in C^\infty(\Omega)$, there are 
  $k\in\N$ and a positive constant $C_{1}>0$ such that
  \beqs
\label{8}
  N^{pm}(u)\leq C_{1}^p\sum_{q=0}^p\binom pq
  \frac{e^{\frac1k\varphi^*(pmk)}}{e^{\frac1k\varphi^*(qmk)}}
  N^0(P^qu)
  \eeqs
    for every $p\in\N_0$.
\end{Prop}

\begin{proof}
  Let us proceed by induction on $p$.

  For $p=0$ it's trivial. Let us assume \eqref{8} to be true for
  $0,1,\ldots,p-1$ and let us prove it for $p$.

  Applying \eqref{9} for $q\in\{1,\ldots,p-1\}$ instead of $p$, we have that
  \beqsn
  &&N^m(u)\leq  C_{0}\left\{N^0(Pu)+e^{\frac1k\varphi^*\left(mk\right)}N^0(u)
  \right\}\\
  &&\ \vdots\\
  &&N^{(p-1)m}(u)\leq C_{0}\left\{N^{(p-2)m}(Pu)+\sum_{q=0}^{p-2}
  \frac{e^{\frac1k\varphi^*\left((p-1)mk\right)}}{e^{\frac1k\varphi^*\left(qmk\right)}}
  N^{qm}(u)\right\}.
  \eeqsn
Substituting in \eqref{9}:
  \beqsn
  && N^{pm}(u)\\\nonumber&&\leq C_{0}\bigg\{N^{(p-1)m}(Pu)+
  \frac{e^{\frac1k\varphi^*\left(pmk\right)}}{e^{\frac1k\varphi^*\left((p-1)mk\right)}}
  N^{(p-1)m}(u)+\ldots+e^{\frac1k\varphi^*\left(pmk\right)}N^0(u)\bigg\}\\
  &&\leq C_{0}\bigg\{N^{(p-1)m}(Pu)+
  \frac{e^{\frac1k\varphi^*\left(pmk\right)}}{e^{\frac1k\varphi^*\left((p-1)mk\right)}}
  C_{0}\bigg[N^{(p-2)m}(Pu)\\&&\nonumber\qquad+
  \frac{e^{\frac1k\varphi^*\left((p-1)mk
\right)}}{e^{\frac1k\varphi^*\left((p-2)mk\right)}}
  N^{(p-2)m}(u)+\\
  &&\qquad\ldots+e^{\frac1k\varphi^*\left((p-1)mk\right)}N^0(u)\bigg]
  +  \ldots+e^{\frac1k\varphi^*\left(pmk\right)}N^0(u)\bigg\}\\
  &&\leq C_{0}N^{(p-1)m}(Pu)+C_{0}^2
  \frac{e^{\frac1k\varphi^*\left(pmk\right)}}{e^{\frac1k\varphi^*\left((p-1)mk\right)}}
  N^{(p-2)m}(Pu)\\&&\qquad  +C_{0}^2
  \frac{e^{\frac1k\varphi^*\left(pmk\right)}}{e^{\frac1k\varphi^*\left((p-2)mk\right)}}
  N^{(p-2)m}(u)+
  \ldots+C_{0}(C_{0}+1)e^{\frac1k\varphi^*\left(pmk\right)}N^0(u)\\
&&\vdots\ \\
 &&\leq \sum_{q=0}^{p-1}
  \frac{e^{\frac1k\varphi^*\left(pmk\right)}}{e^{\frac1k\varphi^*\left((q+1)mk\right)}}
  C_{0}^{p-q}N^{qm}(Pu)+
  (C_{0}+1)^pe^{\frac1k\varphi^*\left(pmk\right)}N^0(u)\\
 &&\leq \sum_{q=0}^{p-1}
  \frac{e^{\frac1k\varphi^*\left(pmk\right)}}{e^{\frac1k\varphi^*\left((q+1)mk\right)}}
  C_{1}^{p-q}N^{qm}(Pu)+
  C_{1}^pe^{\frac1k\varphi^*\left(pmk\right)}N^0(u)
  \eeqsn
  with $C_{1}:=C_{0}+1$.

  Therefore, by the induction assumption
  and because of property $(h)$ of Proposition~\ref{prop1},
  \beqs
  \nonumber
  N^{pm}(u)\leq&&\sum_{q=0}^{p-1}
  \frac{e^{\frac1k\varphi^*\left(pmk\right)}}{e^{\frac1k\varphi^*
\left((q+1)mk\right)}}
  C_{1}^{p-q}C_{1}^q\sum_{r=0}^q\binom qr
  \frac{e^{\frac1k\varphi^*\left(qmk\right)}}{e^{\frac1k\varphi^*\left(rmk
\right)}}
  N^0(P^rPu)\\
  \nonumber
  &&+C_{1}^pe^{\frac1k\varphi^*\left(pmk\right)}N^0(u)\\
  \label{eq5}
  \leq&&C_{1}^p\sum_{r=0}^{p-1}\sum_{q=r}^{p-1}
  \frac{e^{\frac1k\varphi^*\left(pmk\right)}}{e^{\frac1k\varphi^*\left((r+1)mk\right)}}
  \binom qr N^0(P^{r+1}u)\nonumber\\&&\qquad 
+C_{1}^pe^{\frac1k\varphi^*\left(pmk\right)}N^0(u).
\eeqs
Let us now remark that $\sum_{q=r}^{p-1}\binom qr=\binom{p}{r+1}$ and hence
substituting in \eqref{eq5}, we finally have:
\beqsn
N^{pm}(u)\leq&&
C_{1}^p\sum_{r=0}^{p-1}\binom{p}{r+1}
\frac{e^{\frac1k\varphi^*\left(pmk\right)}}{e^{\frac1k\varphi^*\left((r+1)mk\right)}}
N^0(P^{r+1}u)
+C_{1}^pe^{\frac1k\varphi^*\left(pmk\right)}N^0(u)\\
=&&C_{1}^p\sum_{r'=0}^p\binom{p}{r'}
\frac{e^{\frac1k\varphi^*\left(pmk\right)}}{e^{\frac1k\varphi^*\left(r'mk\right)}}
N^0(P^{r'}u),
\eeqsn
so that \eqref{8} is valid with $C_{1}=1+C_{0}$.
\end{proof}

\section{Proof of Theorems~\ref{th13bis} and \ref{th13}}

We can now proceed with the proof of Theorem~\ref{th13bis}, where formula
\eqref{15} here below has been modified with respect to formula (3.3) in the first version \cite{BJ}
of the present paper, because of a remark of
Hoepfner and Rampazo who pointed out very recently in \cite{HR} that the dependence on $q$ of the constant $C_q$ in such a formula of \cite{BJ} is crucial in the induction process and that this fact could ruin the induction argument. However, the original proof in \cite{BJ}, which was based on Komatsu~\cite{K3}, is easy to repair (using again the original ideas of \cite{K3}). Below, we clarify the induction argument presenting an alternative, but very similar, proof of the induction argument, using an induction hypothesis adapted to this setting, analogous to that of formula (6) in Komatsu~\cite{K3}. 

\begin{proof}[Proof of Theorem~\ref{th13bis}] 

Let us first prove that if $P(x,D)$ is elliptic then
$\E^P_{\{\omega\}}(\Omega)\subseteq\E_{\{\omega\}}(\Omega)$.

Let $u\in C^\infty(\Omega)$ satisfy \eqref{roumieu} for every 
$K\subset\subset\Omega$. In particular it satisfies \eqref{roumieu} for every
relatively compact subdomain
$G\subset\Omega$. 
From Proposition~\ref{prop10bis}, for every fixed $\delta>0$ and for all 
$p\in\N_0$
\beqs
\nonumber
\|\nabla^{pm}u\|_\delta\leq&&\delta^{-pm}N^{pm}(u)
\leq\delta^{-pm}C_{1}^p\sum_{q=0}^p\binom pq
\frac{e^{\frac1k\varphi^*\left(pmk\right)}}{e^{\frac1k\varphi^*\left(qmk\right)}}
N^0(P^qu)\\
\nonumber
\leq&&\delta^{-pm}C_{1}^p\sum_{q=0}^p\binom pq
\frac{e^{\frac1k\varphi^*\left(pmk\right)}}{e^{\frac1k\varphi^*\left(qmk\right)}}
\|P^qu\|_{L^2(G)}\\
\nonumber
\leq&&\delta^{-pm}C_{1}^p\sum_{q=0}^p\binom pq
\frac{e^{\frac1k\varphi^*\left(pmk\right)}}{e^{\frac1k\varphi^*\left(qmk\right)}}
ce^{\frac1k\varphi^*\left(qmk\right)}\\
\nonumber
\leq&&c(\delta^{-1}C_{1}^{1/m}2^{1/m})^{pm}
e^{\frac1k\varphi^*\left(pmk\right)}\\
\label{12}
\leq&&cD_{\delta}\,e^{\frac{1}{k'}\varphi^*\left(pmk'\right)}
=\tilde{C}e^{\frac{1}{k'}\varphi^*\left(pmk'\right)}
\eeqs
for some $k'\in\N$, $D_{\delta}\,, \tilde{C}>0$, because of $(e)$ of
Proposition~\ref{prop1}.

By \eqref{K3} (with $\sigma=\delta$, $\varepsilon=1$, $f=\nabla^{pm}u$), and by
\eqref{12},
for all $1\leq t\leq m-1$, $t'=m-t$, $q=pm+t$ we have, by the convexity of 
$\varphi^*$:
\beqs
\nonumber
\|\nabla^qu\|_{2\delta}=&&\|\nabla^{pm+t}u\|_{2\delta}=
\|\nabla^{m-t'}\nabla^{pm}u\|_{2\delta}\\
\nonumber
\leq&&C\left(\|\nabla^{(p+1)m}u\|_\delta+(\delta^{-m}+1)\|\nabla^{pm}u\|_\delta
\right)\\
\nonumber
\leq&&C\tilde{C}\left[e^{\frac{1}{k'}\varphi^*\left((p+1)mk'\right)}
  +(\delta^{-m}+1)e^{\frac{1}{k'}\varphi^*\left(pmk'\right)}\right]\\
\nonumber
\leq&&C\tilde{C}(2+\delta^{-m})e^{\frac{1}{k'}\varphi^*\left(((p+1)m+t)k'\right)}\\
\nonumber
\leq&&C\tilde{C}(2+\delta^{-m})e^{\frac{1}{2k'}\varphi^*\left(2(pm+t)k'\right)}
e^{\frac{1}{2k'}\varphi^*\left(2mk'\right)}\\
\label{13}
=&&C_{\delta}e^{\frac{1}{k''}\varphi^*\left(qk''\right)}
\eeqs
for $C_{\delta}=C\tilde{C}(2+\delta^{-m})e^{\frac{1}{2k'}\varphi^*\left(2mk'\right)}$
and $k''=2k'$.

From \eqref{12} and \eqref{13}, and by Sobolev inequality 
(cf. \cite[Lemma 2.5]{KG}), we thus have that 
$u\in\E_{\{\omega\}}(G_{2\delta})$
for every fixed $\delta>0$ and hence $u\in\E_{\{\omega\}}(\Omega)$.

Let us now show (i).
Let $u\in\E_{\{\omega\}}(\Omega)$ and prove by induction on $p$ that 
for every $K\subset\subset\Omega$
there 
exists $k\in\N$ and $B>1$ such that for every  $p,q\in\N_0$ we have
\beqs
\label{15}
\|\nabla^qP^pu\|_{L^2(K)}\leq B^{q+pm+p+1}e^{\frac1k\varphi^*\left((q+pm)k\right)}.
\eeqs

Indeed, for $p=0$, formula  \eqref{15} is valid because $u\in\E_{\{\omega\}}(\Omega)$.
Assume that \eqref{15} is true for $p$, and all $q\in\N_0$, and prove it
for $p+1$. Since the coefficients of $P$ are in $\E_{\{\omega\}}(\Omega)$ we have
\beqs
\nonumber
&&\|\nabla^qP^{p+1}u\|_{L^2(K)}=\sum_{r=0}^q\binom qr\|P^{[r]}\nabla^{q-r}P^pu\|_{L^2(K)}\\
\nonumber
&&\quad\leq\sum_{r=0}^q\binom qr A^{r+1}e^{\frac1k\varphi^*\left(rk\right)}
\sum_{s=0}^m\|\nabla^{q+s-r}(P^pu)\|_{L^2(K)}\\
\nonumber
&&\quad =A^{r+1}\sum_{r=0}^q\frac{q!}{(q-r)!}a_{r,\frac1k}\|\nabla^{q+m-r}(P^pu)\|_{L^2(K)}\\
\label{eq7}
&&\ \ \ \ \ \quad +A^{r+1}\sum_{r=0}^q
\frac{q!}{r!(q-r)!}e^{\frac1k\varphi^*\left(rk\right)}
\sum_{s=0}^{m-1}\|\nabla^{q+s-r}(P^pu)\|_{L^2(K)},
\eeqs
for some constant $A>0$ which depends only on the coffeicients of $P(x,D)$.

By property $(a)$ of Proposition~\ref{prop1} we have, for $0\leq r\leq q$,
\beqsn
\frac{q!}{(q-r)!}a_{r,\frac1k}\leq
\frac{q!}{(q-r)!}a_{r,\frac1k}a_{q-r,\frac1k}\frac{(q-r)!}{e^{\frac1k\varphi^*((q-r)k)}}
\leq q!a_{q,\frac1k}e^{-\varphi^*(q-r)}
\leq q!a_{q,\frac1k}e^{-\varphi^*(0)}.
\eeqsn
and hence, using the last inequality in \eqref{eq7} and separating the derivatives
$\nabla^\sigma(P^pu)$ for $\sigma\geq m$ and
$0\leq\sigma\leq m-1$ we obtain
\beqsn
\|\nabla^qP^{p+1}u\|_{L^2(K)}\leq
&&(m+1)\sum_{r=0}^q\frac{q!}{(q-r)!}A^{r+1}a_{r,\frac1k}
\|\nabla^{q+m-r}(P^pu)\|_{L^2(K)}\\
&&\ \ \ +me^{-\varphi^*(0)}
A^{q+1}q!a_{q,\frac1k}\sum_{\sigma=0}^{m-1}\|\nabla^\sigma(P^pu)\|_{L^2(K)}.
\eeqsn

Now, by the induction hypothesis \eqref{15} and property $(a)$ of
Proposition~\ref{prop1} we deduce
\beqsn
&&\|\nabla^qP^{p+1}u\|_{L^2(K)}\leq
\frac{(m+1)A}{B}\Big[\sum_{r=0}^q\frac{q!}{(q-r)!}A^{r}a_{r,\frac1k}
B^{-r}e^{\frac1k\varphi^*\left((q+m-r+pm)k\right)}\\
&&\quad\ \ \ \ \ \ \ \ \ \ \ \ \ \ \ \ \   +A^qe^{-\varphi^*(0)}
q!a_{q,\frac1k}\sum_{\sigma=0}^{m-1}B^{\sigma-m-q}e^{\frac1k\varphi^*\left((\sigma+pm)k\right)}\Big]\cdot B^{q+(p+1)m+p+2}\\
&&\quad= \frac{(m+1)A}{B}\Big[\sum_{r=0}^q\frac{q!}{(q-r)!}A^{r}a_{r,\frac1k}
B^{-r}a_{q+(p+1)m-r,\frac1k}(q+(p+1)m-r)!\\
&&\quad\ \ \ \ \ \ \ \ \ \ \ \  +A^qe^{-\varphi^*(0)}
q!a_{q,\frac1k}\sum_{\sigma=0}^{m-1}B^{\sigma-m-q}a_{\sigma+pm,\frac1k}(\sigma+pm)!\Big]\cdot B^{q+(p+1)m+p+2}\\
&&\quad\leq  \frac{(m+1)\max\{1,e^{-(m+1)\varphi^*(0)}\}A}{B}\Bigg[\sum_{r=0}^q\frac{q!}{(q-r)!}\frac{(q+(p+1)m-r)!}{(q+(p+1)m)!}A^rB^{-r}\\
&&\qquad +\sum_{\sigma=0}^{m-1}\frac{q!(\sigma+pm)!}{(q+(p+1)m)!}A^qB^{-q}\Bigg]\cdot B^{q+(p+1)m+p+2}e^{\frac1k\varphi^*\left((q+(p+1)m)k\right)},
\eeqsn
because of $a_{q+\sigma+pm,\frac1k}\leq e^{-(m-\sigma)\varphi^*(0)}a_{q+(p+1)m,\frac1k}
\leq\max\{1,e^{-m\varphi^*(0)}\}a_{q+(p+1)m,\frac1k}$.

Now, since
\beqsn
\frac{q!}{(q-r)!}\frac{(q+(p+1)m-r)!}{(q+(p+1)m)!}=\frac{\binom{q}{r}}{\binom{q+(p+1)m}{r}}
\leq1,
\eeqsn
and
\beqsn
\frac{q!(\sigma+pm)!}{(q+(p+1)m)!}\leq\frac{1}{\binom{q+(p+1)m}{q}}\leq1,
\eeqsn
it is sufficient to take $B\geq\max\{1,e^{-(m+1)\varphi^*(0)}\}(m+2)^2A$ so that the factor in the brackets does not exceed $m+2$ for any $q$ since $A/B<1/2$ and hence all the right hand side is less than or equal to $B^{q+(p+1)m+p+2}e^{\frac1k\varphi^*\left((q+(p+1)m)k\right)}$.

Therefore \eqref{15} is proved by induction and, in particular,
\eqref{roumieu} holds for $q=0$.
An application of Proposition~\ref{prop1}(e) gives the conclusion.
\end{proof}

\begin{proof}[Proof of Theorem~\ref{th13}]
The proof of $(i)$ is similar to the Roumieu case, Theorem~\ref{th13bis}(i), for
$C_{q,k}$ and $c_k$ instead of $C_q$ and $c$.

However, since the
constant $C_1$ of \eqref{8} depends on $k$, we cannot deduce formula
\eqref{12} from $(e)$ of Proposition~\ref{prop1}. To prove $(ii)$ we first remark that $\E_{\{\omega\}}(\Omega)
\subseteq\E_{(\sigma)}(\Omega)$ for $\sigma(t)=o(\omega(t))$ as $t\to\infty$ by \cite[Prop. 4.7]{BMT}.
Therefore by 
Theorem~\ref{th13bis}(ii) we have
\beqsn
\E^P_{(\omega)}(\Omega)\subseteq\E^P_{\{\omega\}}(\Omega)\subseteq
\E_{\{\omega\}}(\Omega)\subseteq\E_{(\sigma)}(\Omega)
\eeqsn
which concludes the proof in the Beurling case.
\end{proof}

We conclude proving that ellipticity is  necessary 
in Theorems~\ref{th13bis}(ii) and \ref{th13}(ii):

\begin{Ex}
\label{ex32}
\begin{em}
Let $P(x,D)$ be a linear partial differential operator with real analytic 
coefficients of order $m$ not elliptic
in $(x_0,\xi_0)\in\Omega\times\R^n$, for a domain $\Omega\subseteq\R^n$
and $\|\xi_0\|=1$, i.e.
\beqsn
P_m(x_0,\xi_0)=0,
\eeqsn
where $P_m$ is the principal part of $P$.

We are going to prove that there exist a function $u$ and
a subadditive weight $\omega$, which is not a Gevrey weight in general and is between two given Gevrey weights, and such that 
$u\in\E^P_{\{\omega\}}(\Omega)\setminus\E_{\{\omega\}}(\Omega)$, and that
$u\in\E^P_{(\omega)}(\Omega)\setminus\E_{(\sigma)}(\Omega)$, for some subadditive weight function $\sigma=o(\omega)$. Consequently, the ellipticity 
of $P$ is
needed for statement $(ii)$ of Theorems~\ref{th13bis} and \ref{th13}.

To construct $\omega$ and the function $u$ we follow \cite{metivier1978propiete}:
for any fixed $s>1$ we choose $\sigma\in(1,s)$ and $\varepsilon>0$ such that
\beqsn
0<\varepsilon<\frac{m(s-\sigma)}{2ms-\sigma}<\frac12\,.
\eeqsn
Then we take $\delta>0$ so that $B(x_0,2\delta)\subset\subset\Omega$ and
$\varphi\in\E_{(t^{1/\sigma})}(\R^n)$ with $\supp\varphi\subset B(0,2\delta)$.
For $\eta=\frac{m-\varepsilon}{ms}$ we finally define, as in \cite{metivier1978propiete},
\beqsn
u(x):=\int_1^{+\infty}\varphi\big(\rho^\varepsilon(x-x_0)\big)e^{-\rho^\eta}
e^{i\rho\langle x-x_0,\xi_0\rangle}d\rho\,.
\eeqsn

It was proved in \cite{metivier1978propiete} that
\beqs
\label{M24}
(D_{\xi_0}^\alpha u)(x_0)=\frac1\eta\Gamma\left(\frac{\alpha+1}{\eta}\right)
+o(1),
\eeqs
where $\Gamma$ is the gamma function, so that $u\notin\E_{\{t^{1/s'}\}}(U)$
in any neighborhood $U$ of $x_0$ for any $s'<1/\eta$ (nor, in particular, for
$s'=s$), but $u\in\E_{\{t^\eta\}}(\R^n)$.
Moreover, it was proved in \cite{metivier1978propiete} that $u\in\E_{\{t^{1/s}\}}^P(\Omega)$.

Let us now consider any subadditive weight function $\omega(t)$ such that
$\omega(t)=o(t^{1/s})$ and $t^{1/s'}=o(\omega(t))$ for $s'>s>1$. For instance,
$\omega(t)=t^{1/s}/\log t$. In general, such a weight exists by \cite[Proposition 1.9]{BMT}.

We have that 
$\E_{(\omega)}(\Omega)\subseteq\E_{\{\omega\}}(\Omega)\subseteq
\E_{\{t^{1/s'}\}}(\Omega)$ 
and 
$\E_{\{t^{1/s}\}}(\Omega)\subseteq\E_{(\omega)}(\Omega)
\subseteq\E_{\{\omega\}}(\Omega)$ by \cite[Prop. 4.7]{BMT}. Analogously $\E^P_{\{t^{1/s}\}}(\Omega)\subseteq\E^P_{(\omega)}(\Omega)
\subseteq\E^P_{\{\omega\}}(\Omega)$, 
so that $u\in\E^P_{\{\omega\}}(\Omega)\setminus\E_{\{\omega\}}(\Omega)$ and 
ellipticity is necessary in Theorem~\ref{th13bis}~(ii).

Moreover, if $\sigma(t):=t^{1/s'}$ we clearly have $u\in\E^P_{(\omega)}(\Omega)\setminus\E_{(\sigma)}(\Omega)$.
Since $\sigma(t)=o(\omega(t))$ as $t\to\infty$, this proves that ellipticity is necessary in
Theorem~\ref{th13}~(ii).

\end{em}
\end{Ex}

{\bf Acknowledgments:}
The research of the second author was partially supported by the project PID2020-119457GB-100 funded by MCIN/AEI/10.13039/501100011033 and by ``ERDF A way of making Europe'', and by the project GV AICO/2021/170.


\begin{thebibliography}{99}

\bibitem{BJ1}
C.~Boiti, D.~Jornet, {\em The problem of iterates in some classes of
ultradifferentiable functions}, ``Operator Theory:
Advances and Applications'', Birkhauser, Basel, {\bf 245} (2015), 21-33.

\bibitem{BJ2}
C.~Boiti, D.~Jornet, {\em A characterization of the wave front set defined
by the iterates of an operator with constant coefficients}, 
RACSAM - Revista de la Real Academia de Ciencias Exactas, F\'isicas 
y Naturales. Serie A. Matem\'aticas, {\bf 111}, n.3 (2017), 891-919.

\bibitem{BJ}
C.~Boiti, D.~Jornet, {\em A simple proof of Kotake-Narasimhan theorem in some classes of ultradifferentiable functions}, J. Pseudo-Differ. Oper. Appl. 8 (2017), no. 2, 297-317. 

\bibitem{BJJ}
C.~Boiti, D.~Jornet, J.~Juan-Huguet, {\em Wave front set with
respect to the iterates of an operator with constant coefficients},
Abstr. Appl. Anal. { 2014}, Article ID 438716 (2014), pp. 1-17,
http://dx.doi.org/10.1155/2014/438716


\bibitem{BCM} P.~Bolley, J.~Camus,
C.~Mattera, {\em Analyticit\'e microlocale et it\'er\'es d'operateurs
hypoelliptiques}, S\'eminaire Goulaouic-Schwartz, 1978-79, Exp N.13,
\'Ecole Polytech., Palaiseau.


\bibitem{bonet_meise_melikhov2007a} J. Bonet, R. Meise, S.N. Melikhov,
{\em A comparison of two different ways of define classes of
ultradifferentiable functions}, Bull. Belg. Math. Soc. Simon Stevin 14 (2007),
425-444.

\bibitem{BMT} R.W.~Braun, R.~Meise, B.A.~Taylor,
{\em Ultradifferentiable functions and Fourier analysis}, Result. Math.
{ 17} (1990), 206-237.


\bibitem{FG} C. Fern\'andez, A. Galbis,  {\em Superposition in classes of ultradifferentiable functions}, Publ. Res. I. Math. Sci., 42(2) (2006), 399-419.

\bibitem{HR}
G.~Hoepfner, P.~Rampazo, {\em The global Kotake-Narasimhan theorem}, Proc. Amer. Math. Soc. 150 (2022), no. 3, 1041-1057.

\bibitem{david}
D.~Jornet, {\em Operadores Pseudodiferenciales en Clases no Casianal\'iticas
de tipo Beurling}, Universidad Polit\'ecnica de Valencia, Ph-D Thesis (2003). 

\bibitem{H}
J.~Juan-Huguet,~{\em Iterates and
Hypoellipticity of Partial Differential Operators on Non-Quasianalytic
Classes}, Integr. Equ. Oper. Theory {\bf 68} (2010), 263-286.

\bibitem{J} J.~Juan-Huguet, {\em A Paley-Wiener type theorem for
generalized non-quasianalytic classes}, Studia Math. { 208}, n.1 (2012),
31-46.

\bibitem{K1} H. Komatsu, {\em A characterization of
real analytic functions}, Proc. Japan Acad. {\bf 36} (1960), 90-93.

\bibitem{K2} H. Komatsu, {\em On interior regularities of the solutions 
of principally elliptic systems
of linear partial differential equations}, J. Fac. Sci., Univ. Tokyo, 
Sec. 1, {\bf 9} (1961), 141-164.

\bibitem{K3} H. Komatsu, {\em A Proof of Kotak\'e and Narasimhan's Theorem},
Proc. Japan Acad. {\bf 38}, n. 9 (1962),  615-618.

\bibitem{KN} T.~Kotake, M.S.~Narasimhan,
{\em Regularity theorems for fractional powers of
a linear elliptic operator}, Bull. Soc. Math. France {\bf 90} (1962), 449-471.

\bibitem{KG} 
H.~Kumano-Go, {\em Pseudo-differential operators}
  The MIT Press, Cambridge, London, 1982.

\bibitem{langenbruch1979P} M. Langenbruch,
{\em P-Funktionale und Randwerte zu hypoelliptischen Differentialoperatoren},
Math. Ann. 239(1) (1979), 55-74.


\bibitem{langenbruch1979Fortsetzung} M. Langenbruch,
{\em Fortsetzung von Randwerten zu hypoelliptischen Differentialoperatoren
und partielle Differentialgleichungen}, J. Reine Angew.
Math. 311/312 (1979), 57-79.


\bibitem{langenbruch1985on} M. Langenbruch,
{\em On the functional dimension of solution spaces of hypoelliptic
partial differential operators}, Math. Ann. 272 (1985), 217-229.


\bibitem{langenbruch1987bases} M. Langenbruch,
{\em Bases in solution sheaves of systems of partial differential
equations}, J. reine angew. Math. 373  (1987), 1-36.



\bibitem{LM}
J.L.~Lions, E.~Magenes, {\em Probl\`emes aux limites non homog\`enes 
et applications}, {\bf 3}, Dunod, Paris (1970).

\bibitem{metivier1978propiete} G. M\'etivier, {\em Propri\'et\'e des
it\'er\'es et ellipticit\'e}, Comm. Part. Diff. Eq. 3 {(9)}
(1978), 827-876.

\bibitem{N}
E.~Nelson, {\em Analytic vectors}, Ann. of Math. {\bf 70} (1959), 572-615.

\bibitem{NZ} E. Newberger, Z. Zielezny, {\em The growth
of hypoelliptic polynomials and Gevrey classes}, Proc. Amer. Math. Soc.
{\bf 39}, n. 3 (1973), 547-552.


\bibitem{old} J. Oldrich, {\em Sulla regolarit\`a delle soluzioni delle equazioni lineari ellittiche nelle classi di Beurling}, (Italian) Boll. Un. Mat. Ital. (4) 2 (1969), 183-195. 

\bibitem{PV} H.-J. Petzsche, D. Vogt, {\em Almost analytic extension of ultradifferentiable functions and the boundary values of holomorphic functions}, Math. Ann., 267(1) (1984), 17-35.

\end{thebibliography}
\end{document}